\documentclass[12pt,a4paper]{amsart}
\usepackage{amsmath,amsthm,amssymb}
\usepackage{stmaryrd}
\usepackage{cleveref}
\usepackage{enumerate}
\usepackage[all]{xy}
\usepackage{xcolor}

\usepackage[utf8]{inputenc}

\DeclareMathOperator{\Ann}{\mathrm{Ann}}
\DeclareMathOperator{\Aut}{\mathrm{Aut}}
\newcommand{\C}{\mathbb{C}}
\DeclareMathOperator{\Hom}{\mathrm{Hom}}

\newcommand{\N}{\mathbb{N}}

\newcommand{\R}{\mathbb{R}}
\newcommand{\Z}{\mathbb{Z}}

\newcommand{\bs}{\backslash}

\newcommand{\ddt}{\frac{d}{dt}}
\newcommand{\derham}[1]{\Omega^{#1}}
\newcommand{\epi}{\twoheadrightarrow}

\newcommand{\extalg}{\bigwedge}
\newcommand{\foliation}{\mathcal{F}}
\newcommand{\forus}[1]{}
\newcommand{\from}{\colon}
\newcommand{\heis}{\mathfrak{h}}

\newcommand{\id}{\mathrm{Id}}
\DeclareMathOperator{\im}{\mathrm{im}}
\newcommand{\klie}{\mathfrak{k}}

\newcommand{\lap}{\Delta}
\newcommand{\liealgdual}{\liealg^*}
\newcommand{\liealg}{\mathfrak{g}}
\newcommand{\lie}{\mathcal{L}}

\newcommand{\pr}{\mathrm{pr}}
\DeclareMathOperator{\rank}{\mathrm{rank}}
\newcommand{\sign}{\mathrm{sign}}

\newcommand{\sm}[1]{ \left(\begin{smallmatrix} #1 \end{smallmatrix}\right)}

\newcommand{\tievsky}{\mathcal{T}}
\newcommand{\tileta}{{\eta_\liealg}}
\newcommand{\tr}[1]{\widehat{#1}}
\newcommand{\uc}[1]{\widetilde{#1}}

\newtheorem{theorem}{Theorem}[section]

\newtheorem{proposition}[theorem]{Proposition}
\newtheorem{proposition-definition}[theorem]{Proposition-Definition}
\newtheorem{lemma}[theorem]{Lemma}
\newtheorem{corollary}[theorem]{Corollary}
\newtheorem{claim}[theorem]{Claim}
\theoremstyle{remark}

\newtheorem{remark}[theorem]{Remark}

\theoremstyle{definition}

\title{Nilpotent aspherical Sasakian manifolds}
\author[A. De Nicola]{Antonio De Nicola}
 \address{Dipartimento di Matematica, Università degli Studi di Salerno, Via Giovanni Paolo II 132, 84084 Fisciano, Italy}
 \email{antondenicola@gmail.com}

\author[I. Yudin]{Ivan Yudin}
 \address{University of Coimbra, CMUC, Department of Mathematics, 3001-501 Coimbra, Portugal}
 \email{yudin@mat.uc.pt}

\thanks{This work was partially supported by the Centre for Mathematics of the
University of Coimbra - UIDB/00324/2020, funded by the Portuguese Government through FCT/MCTES.
ADN is a member of GNSAGA - Istituto Nazionale di Alta Matematica}

\begin{document}
\begin{abstract}
We show that every compact aspherical Sasakian manifold with nilpotent
fundamental group is diffeomorphic to a Heisenberg nilmanifold.
\end{abstract}
\maketitle
\section{Introduction}
The interaction between topological constraints and geometric
structures is a  classical topic. If we are interested only in the homotopy type of an
underlying manifold then Algebraic Topology techniques are usually sufficient.
For example, Benson and Gordon in~\cite{bensongordon} show that a compact nilpotent aspherical manifold that admits
a K\"ahler structure is homotopy equivalent to a torus.
It is known that Sasakian manifolds play the same role in contact geometry as Kähler
manifolds in symplectic geometry.
Similarly to the result of Benson and Gordon, a compact nilpotent aspherical manifold that admits a Sasakian
structure is homotopy equivalent to a Heisenberg nilmanifold~\cite{imrn,bazzoni}.

Recall that an \emph{aspherical manifold} is a manifold whose homotopy groups besides the
fundamental group are trivial. A  manifold is called \emph{nilpotent} if its fundamental group is nilpotent.
According to the Borel conjecture any two compact aspherical manifolds with the same
fundamental group should be homeomorphic. This conjecture is proven for a wide class of
groups, including nilpotent groups (cf.~\cite{lueck}).
In particular, a compact nilpotent aspherical manifold that admits a Sasakian
structure is homeomorphic to a Heisenberg nilmanifold.

One of the most surprising results  of the 20th century was Milnor's discovery in~\cite{milnor}  of exotic
spheres, i.e.\ spheres that are homeomorphic but not diffeomorphic to the
standard one. Nowadays there are many known examples of  topological
manifolds that admit non-equivalent smooth structures.

If one fixes a  smooth structure on a topological manifold, which is quite natural
from a differential geometer's point of view, then the problems of existence of
compatible geometric structures on it become much
harder to solve and require \emph{ad hoc} approaches.
For example the conjecture of  Boyer, Galicki and Kollár \cite{bgk} that predicts that every
parallelizable exotic sphere admits a Sasakian structure is still open.

In this paper we study the existence of Sasakian structures on compact aspherical
manifolds with  nilpotent fundamental group.
Among all aspherical smooth manifolds with nilpotent fundamental group the most
studied class is the class of nilmanifolds.
A \emph{nilmanifold} is
 a compact quotient of  a nilpotent Lie group by
a discrete subgroup with the smooth structure inherited from the Lie group.
One can show that a compact aspherical nilpotent manifold is homotopy equivalent
to a nilmanifold.
Then, it follows from  the truthfulness of the Borel conjecture in the nilpotent
case that the manifold is homeomorphic to a nilmanifold.
If it is not diffeomorphic to a nilmanifold, it is called  an
\emph{exotic nilmanifold}.
By~\cite[Lemma~4]{exoticnil} a connected sum of a nilmanifold with an exotic
sphere is always  an exotic nilmanifold.
 Thus exotic nilmanifolds do exist.
Moreover, it is not difficult to show the existence of contact exotic
nilmanifolds.
By the main result of Meckert in \cite{contactsum},
  the connected
sum of two contact manifolds carries a contact structure.
Moreover, it is shown in~\cite{bgk} that there are infinitely many Sasakian
(hence contact) exotic spheres. Thus the connected sum of a contact nilmanifold
and a Sasakian exotic sphere provides an example of a  contact exotic nilmanifold.

Write $H(1,n)$ for the Heisenberg group of dimension $2n+1$.
The main result of the article is the following theorem.
\begin{theorem}
\label{main2}
If $M^{2n+1}$ is a compact aspherical Sasakian manifold with nilpotent fundamental
group,  then $M$ is diffeomorphic to a Heisenberg nilmanifold $\Gamma\bs H(1,n)$, where $\Gamma$
is a lattice in $H(1,n)$ isomorphic to $\pi_1(M)$.
Moreover,
there is a second type deformation of a left-invariant normal almost contact
structure on $\Gamma \bs H(1,n)$, such that $M$ and $\Gamma\bs H(1,n)$
are isomorphic as normal almost contact manifolds.
\end{theorem}
Equivalently, the above result can be stated as the non-existence of
compact Sasakian exotic nilmanifold.
Theorem~\ref{main2} can be seen as an odd-dimensional version of~\cite{cortes,tralle},
where it is shown that every compact Kähler aspherical nilmanifold is
biholomorphic to a complex torus.
As a corollary, Baues and Kamishima derived in~\cite{baues2020} the result of Theorem~\ref{main2} under the stronger
assumption that the Sasakian structure is \emph{regular}. Their approach, based on passing to
the quotient of the Reeb vector field action, cannot be extended to the
non-regular  case.

Baues and Cortés also showed in~\cite{cortes} that
if $X$ is a compact aspherical Kähler manifold with virtually solvable fundamental group, then $X$ is biholomorphic to a finite quotient of a complex torus.
In the same vain we obtain the following.
\begin{corollary}
\label{solvable}
If $M^{2n+1}$ is a compact aspherical Sasakian manifold with (virtually) solvable fundamental
group, then $M$ is diffeomorphic to  a finite quotient of a Heisenberg nilmanifold.
\end{corollary}
\begin{proof}
Let $M^{2n+1}$ be a compact aspherical Sasakian manifold with virtually solvable fundamental
group. Denote by $\overline{ M }$ the finite cover of $M$ with solvable
fundamental group.  The Sasakian structure on $M$ transfers to a Sasakian
structure on $\overline{ M }$.
By the result of Bieri in~\cite{bieri} (see also \cite{bieri_book}),
every solvable group, for which Poincaré duality holds, is torsion-free and
polycyclic. In~\cite{kasuya}, Kasuya showed that if the fundamental group of a
compact Sasakian manifold is polycyclic, then it is virtually nilpotent. Hence
there is a finite (compact) cover $\widetilde{M}$ of $\overline{ M }$, such
that $\pi_1(\widetilde{M})$  is nilpotent. The Sasakian structure on $\overline{
M }$ transfers to a Sasakian structure on $\widetilde{M}$. As
$\widetilde{M}$ is a finite cover of an aspherical manifold $M$, the manifold
$\widetilde{M}$ is also aspherical. By~Theorem~\ref{main2}, the manifold
$\widetilde{M}$ is diffeomorphic to a Heisenberg nilmanifold of dimension
$2n+1$. Hence $M$ is diffeomorphic to a finite quotient of a Heisenberg
nilmanifold.
\end{proof}
\noindent
Our main result and
the similar one for  the Kähler case provide
 evidence that exotic nilmanifolds do not admit as rigid
geometric structures as nilmanifolds.
Another result that points in the same direction was proved in a recent
article~\cite{anosov}, where it is shown
that
compact exotic nilmanifolds  admit no
 Anosov $\Z^r$-action without rank-one factor.

In the above-mentioned paper \cite{baues2020} on locally homogeneous aspherical Sasakian
manifolds, Baues and Kamishima deal with the regular case by showing that any compact
regular aspherical Sasakian manifold with solvable fundamental group is finitely
covered by a Heisenberg nilmanifold, as its Sasaki structure may be deformed to
a locally homogeneous one.

Despite the analogy between Sasakian and Kähler geometry, the proof of
Theorem~\ref{main2} is significantly different from the proofs
in~\cite{cortes,tralle}. The main obstacle to imitate these proofs is that
there is no suitable version of the Albanese map for Sasakian manifolds.

The main stages of the proof are the following.
In Section~\ref{sasakianmanifolds} we discuss two Tievsky models for compact
Sasakian manifolds. One of them is a subcomplex of the de Rham algebra of
$M$ and the other is a quotient of the former one.  Then we prove
that for every compact aspherical Sasakian manifold with nilpotent fundamental
group $\Gamma$, the Malcev envelope $G(\Gamma)$ of $\Gamma$ is isomorphic to the Heisenberg
group (\Cref{heisenberg}).
Further, we show that
there exists a
quasi-isomorphism $\rho \from \Omega^\bullet(G(\Gamma))^{G(\Gamma)} \to \Omega^\bullet(M)$
such that its image lies in the first Tievsky model of $M$ (\Cref{imquasiiso}).

In \Cref{topology}, we develop a theory that permits to mitigate the
absence of an Albanese map for Sasakian manifolds.
The main result of this section implies that
by modifying the above $\rho$, we can assume that it is the restriction of
$h^*\from \Omega^\bullet(G(\Gamma))^\Gamma \to \Omega^\bullet(M)$, where
$h\from M \to \Gamma \bs G(\Gamma)$
is a  homotopy equivalence.

In \Cref{PROOF}, we prove that $h$ is a diffeomorphism.
First, we deduce from \Cref{etaheis}, that there is a left-invariant $1$-form
$\eta_\heis\in \Omega^1(G(\Gamma))$  and a basic function $f$ on $M$ such that
$h^*(\eta_\heis) = \eta +
(df) \circ \varphi$.

In \Cref{subsection:Sasakian}, we introduce the notion of a $\beta$-twisted
map between Sasakian manifolds, which gives a relative version of second type
deformations of Sasakian manifolds. In the same section, we show that a map
$\psi \from X \to Y$
between two Sasakian manifolds is $(df)$-twisted with $f\in
C^\infty(X)$ if and only if the map
\begin{equation*}
\begin{aligned}
\psi_f \from X \times \R & \to Y\times \R\\
(x,t) & \mapsto ( \psi(x), f(x) + t)
\end{aligned}
\end{equation*}
is holomorphic.
We show that $\Gamma\bs G(\Gamma)$ can be
equipped with a Sasakian structure in such way that $h$ becomes $(df)$-twisted
(Claims~\ref{normalalmostcontact} and~\ref{hdftwisted}).

It is straightforward that $h$ is a diffeomorphism if and only if $h_f$ is a
diffeomorphism.
It is also not very difficult  to show that $h_f$ is surjective, universally closed and proper.
Then using the embedded Hironaka resolution of singularities we show that $h_f$ is a finite map.
This, combined with  several deep
results from complex analytic geometry, implies that $h_f$  is a
biholomorphism, and thus a diffeomorphism.

The paper is organized as follows. Section~\ref{prelim} contains the necessary
preliminaries about Sasakian manifolds, Sullivan CDGAs and nilmanifolds.
In Section~\ref{topology} we prove a rather general result on maps from a
manifold onto its aspherical nilpotent approximations.
In Section~\ref{sasakianmanifolds}  we discuss Tievsky models for compact
Sasakian manifolds, and establish the existence of the quasi-isomorphism
$\rho$ discussed above.  In the final section we complete the
proof of Theorem~\ref{main2}.

\section{Preliminaries} \label{prelim}
\subsection{\it Frölicher-Nijenhuis calculus} \label{FNc}
For a general treatment of
Frölicher-Nijenhuis calculus, we refer  to~\cite{kolarbook}. Given a smooth manifold
$M$ and a vector bundle $E$ over $M$, for every $\psi \in \Omega^k(M, TM)$ one defines an operator
$i_\psi\colon \Omega^\bullet(M, E)\to \Omega^\bullet(M, E) $ of degree $k-1$.
If $\xi$ is a vector field and $\phi$ is an endomorphism of $TM$, the general
definition specializes to
\begin{equation*}
\begin{aligned}
& i_\xi \omega (X_1,\dots, X_{p-1})  = \omega (\xi, X_1, \dots, X_p)\\
& i_\phi \omega (X_1,\dots, X_p)  = \sum_{j=1}^p \omega ( X_1, \dots,\phi X_j,
\dots,  X_p),
\end{aligned}
\end{equation*}
where $\omega \in \Omega^p(M, E)$. In the particular case when $E$ is the
trivial one-dimensional vector bundle over $M$, we get operators $i_\psi$ on
$\Omega^\bullet(M)$.  Next, we define the operators $\lie_\psi$ on
$\Omega^\bullet (M)$ by $\lie_\psi := i_\psi d + (-1)^{k} d i_\psi $.
The Frölicher-Nijenhuis bracket on $\Omega^\bullet(M, TM)$ is defined by the
characteristic property $\left[ \lie_{\psi_1}, \lie_{\psi_2} \right] =
\lie_{[\psi_1,\psi_2]_{FN}}$. Here on the left side $[\ ,\ ]$ stands for the
graded commutator of operators.
\subsection{\it Sasakian manifolds}
\label{subsection:Sasakian}
\emph{An almost contact structure} on a manifold $M$ is a triple $(\varphi, \xi,
\eta)$ where $\varphi\in \Omega^1(M, TM)$, $\xi$ is a vector field and $\eta \in
\Omega^1(M)$ such that $\varphi^2 = - \id + \xi \otimes \eta$ and $\eta(\xi)=1$.
Given an almost contact structure on $M$, one can define an almost complex
structure $J$ on $M \times \R$ by
\begin{equation}
\label{Jalmostcontact}
\begin{aligned}
J\left(X, a\ddt\right)  = \left( \varphi X - a \xi , \eta(X) \ddt \right).
\end{aligned}
\end{equation}
For every almost contact structure one has
\begin{equation}
\label{etaphixi}
\begin{aligned}
\eta\circ \varphi =0,\quad \varphi \xi =0.
\end{aligned}
\end{equation}

We say that $(\varphi, \xi, \eta)$ is \emph{normal} if $J$ is integrable.
It can be verified by simple computation (cf.~\cite[Sec.~6.1]{blair})
that the above definition is equivalent to the vanishing of four tensors:
\begin{equation}
\label{sastensors}
\begin{aligned}
\lie_\xi \eta=0, \quad \lie_\xi\varphi=0,\quad \lie_\varphi \eta=0,\quad
[\varphi,\varphi]_{FN} + 2 d\eta\otimes \xi=0.
\end{aligned}
\end{equation}
Actually, it can be shown that the vanishing of $\left[ \varphi, \varphi
\right]_{FN} + 2d\eta \otimes \xi$ implies that $(\varphi, \xi, \eta)$ is
normal.

On every normal almost contact  manifold, one also  has
\begin{equation}
\label{iphideta}
\begin{aligned}
i_\varphi d\eta =0,\quad i_\xi d\eta =0.
\end{aligned}
\end{equation}
For a normal almost contact manifold  $(M,\varphi, \xi,\eta)$ define the symmetric
$2$-tensor $g$ by
\begin{equation*}
\begin{aligned}
g (X,Y) = \frac12 d\eta(\varphi X, Y) + \eta(X) \eta(Y).
\end{aligned}
\end{equation*}
A normal almost contact manifold $(M,\varphi, \xi, \eta)$
is called \emph{Sasakian} if $g$ is a Riemannian metric.
It can be checked (cf. \cite[Section~6.5]{blair}) that $M$ is Sasakian if and only
if $\tilde{g} := e^{2t}( g + dt^2)$ and $J$ defined by~\eqref{Jalmostcontact}
give a Kähler structure on $M \times \R$.
\begin{remark}
\label{exact} The corresponding Kähler  form $e^{2t}(\Phi + dt
\wedge \eta ) $ is exact. Indeed it equals to one-half of $d ( e^{2t}  \eta)$.
\end{remark}
Following the development in~\cite[Sec.~7.5.1]{galicki}, we say that
an almost contact structure $(\varphi',
\xi,\eta')$ is a \emph{second type deformation} of an almost contact structure
$(\varphi,\xi,\eta)$ on $M$, if  there is a basic $1$-form $\beta$ on
$M$ such that  $\eta' = \eta - i_\varphi \beta$ and $\varphi'
= \varphi - \xi \otimes \beta$.

Let $M$ and $N$ be an almost contact manifolds and $\beta$ a basic $1$-form
on $M$.
We say that a smooth map $h\from M \to N$ is \emph{$\beta$-twisted} if
 \begin{equation*}
\begin{aligned}
& \eta_M =  h^*\eta_N - i_{\varphi_M} \beta,\quad Th\circ \xi_M = \xi_N \circ
h,
\\
&  Th \circ ( \varphi_M + \xi_M \otimes \beta)
  = \varphi_N \circ  Th.
\end{aligned}
\end{equation*}
This definition is designed so that
if $h$ is a  $\beta$-twisted diffeomorphism, then the almost contact structure on
$N$, transferred from $M$ via $h$, is a second type deformation of $(\varphi_N,
\xi_N, \eta_N)$.
Also, if the identity map from $(M,\varphi',\xi,\eta')$ to $(M,\varphi,\xi,\eta)$ is
$\beta$-twisted, then $(\varphi',\xi,\eta')$ is
a second type deformation of $(\varphi, \xi,\eta)$.

For $h\from M \to N$ and $f\in C^\infty(M)$, define
\begin{equation}
\label{hf}
\begin{aligned}
h_f \from M\times \R & \to N\times \R \\
(x,t) & \mapsto ( h(x) , f(x) +t).
\end{aligned}
\end{equation}
\begin{proposition}
\label{hfholomorphicnew}
Let $M$ and $N$ be almost contact manifolds, $h\from M\to N$ a smooth
map, and $f$ a smooth function on $M$. The map $h_f$ is holomorphic if and only if
$f$, and hence also $df$, is basic  and $h$ is $df$-twisted.
\end{proposition}
\begin{proof}
For $L \in \left\{ M, N \right\}$ write the elements of $T(L\times \R)$ in the
form $\binom{X}{Y}$ with $X \in TL$ and $Y\in T\R$. Then
for $x\in M$, $y\in N$, $t\in \R$
\begin{equation*}
\begin{aligned}
J_{M\times \R,(x,t)} = \sm{ \varphi_{M,x} & -\xi_{M,x}
\otimes dt
\\ \ddt\otimes \eta_{M,x} & 0 },
 \quad
J_{N\times \R,(y,t)} = \sm{ \varphi_{N,y} & -\xi_{N,y}\otimes dt
\\ \ddt \otimes \eta_{N,y} & 0 }.
\end{aligned}
\end{equation*}
and
\begin{equation*}
\begin{aligned}
T_{(x,t)} h_f = \sm{T_x h & 0 \\ T_x f & T_t \tau_{f(x)} },
\end{aligned}
\end{equation*}
where $\tau_s \from \R\to \R$ is the translation by $s$.
Now $Th_f \circ J_{M\times \R}$ equals to  $J_{N\times \R}\circ Th_f$ at $(x,t)$
 if and only if
\begin{equation*}
\begin{aligned}
T_x h \circ \varphi_{M,x} & = \varphi_{N,h(x)}  \circ T_x h - \xi_{N,h(x)}
\otimes (df)_x,&
T_x h (\xi_{M,x})  & = \xi_{N,h(x)}  \\
 T_x f \circ \varphi_M & + \ddt\otimes \eta_{M,x} = \ddt\otimes \eta_{N,h(x)}  \circ T_x h, &
T_x f(\xi_{M,x}) & =0.
\end{aligned}
\end{equation*}
The last equation is equivalent to $f$ being basic. The first three equations
are equivalent to the definition $df$-twisted map.
\end{proof}
Let $(M,\varphi,\xi,\eta)$ be a normal almost contact manifold and $\foliation$
be the one-dimensional foliation on $M$ generated by $\xi$.
The manifold $M$ is transversely complex
with respect to $\foliation$.
Namely, for each foliated chart $U$ of $(M,\foliation)$, the endomorphism
$\varphi$ of $TM$ induces a complex structure on the tangent bundle of the leaf space $U/\foliation$.
Write $\pi_U$ for the projection from $U$ onto $U/\foliation$.
The induced complex structure $J$ on $U/\foliation$ is uniquely characterized by
$T\pi_U (\varphi X) = J (T\pi_U (X))$ for all $x\in U$ and $X \in T_x U$.

If $U$ is sufficiently small, the set $U/\foliation$ admits holomorphic
coordinates $z_1$,\dots, $z_n$.
Using these coordinates one can introduce the operators $\partial$ and
$\bar\partial$ on the complexified de Rham complex $
\Omega^\bullet(U/\foliation)_\C$ in the usual way.
The operator $d^c$ on $\Omega^\bullet(U/\foliation)_\C$ is frequently defined as
$d^{c}= i(\bar\partial - \partial)$.
It should be noticed that it can be identified with $(-1)\lie_J$.
In particular, $d^c$ preserves $\Omega^\bullet(U/\foliation)$.

The pull-back map $\pi_U^*$ induces an isomorphism between
$\Omega^\bullet(U/\foliation)$ and the basic de Rham complex
$\Omega^\bullet_B(U,\foliation)$.
These isomorphisms permit to glue the operators $d^c$ defined on
$\Omega^*(U/\foliation)$ for different $U$ to a globally defined operator
$d_B^c$ on $\Omega^\bullet_B(M,\foliation)$. The operator $d_B^c$ is uniquely
characterized by the property
\begin{equation*}
\begin{aligned}
 ( d_B^c)( \omega|_U) = \pi_U^* \left( d^c \left( (\pi_U^*)^{-1} \omega|_U
\right) \right)
\end{aligned}
\end{equation*}
for all $\omega \in \Omega_B^\bullet(M,\foliation)$ and sufficiently small
foliated charts $U$.
The computation
\begin{equation*}
\begin{aligned}
( i_\varphi \pi_U^* \omega) (X) = \omega ( (T\pi_U) ( \varphi (X)) ) =
\omega ( J ((T\pi_U) (X))) = \pi_U^* (i_J \omega) (X)
\end{aligned}
\end{equation*}
implies
that $\pi_U^* \circ d^c = - \pi_U^* \circ \lie_J = -
\lie_\varphi \circ \pi_U^*$.
 We see that $d^c_B$ coincides with the restriction of $(-1)\lie_\varphi$ to
basic forms.

\subsection{\it Sullivan models}
A~\emph{differential graded algebra} is
a graded algebra
$A = \bigoplus_{k\ge 0} A_k$ equipped with a derivation $d$ of degree one such
that $d^2=0$. It is \emph{commutative}
if
$ab = (-1)^{k\ell} ba$,
for all $a \in A_k$ and~$b\in A_\ell$.

The motivating example of a commutative differential graded algebra (CDGA) is provided by the de Rham algebra
$\Omega^\bullet(M)$ of a manifold $M$.
For CDGAs $A$ and $B$, a \emph{homomorphism} $f\colon A\to B$ of CDGAs is a
degree preserving homomorphism of algebras which commutes with $d$. Each
homomorphism of CDGAs $f\colon A\to B$ induces a homomorphism
$H^\bullet(f) \colon H^\bullet(A) \to H^\bullet(B)$ of graded algebras by
$H^\bullet(f)([a]) = [f(a)]$ for $a\in A$ such that $da =0$.
We say that a homomorphism of CDGAs $f\colon A\to B$ is a quasi-isomorphism if
$H^\bullet(f)$ is an isomorphism.
Two CDGAs $A$ and $B$ are said to be \emph{quasi-isomorphic} if there are  CDGAs
\begin{equation*}
\begin{aligned}
A_0=A, A_1, \dots, A_{2k} = B,
\end{aligned}
\end{equation*}
and quasi-isomorphisms $f_j\colon A_{2j+1} \to A_{2j}$, $h_j\colon A_{2j+1} \to
A_{2j+2}$ for $j$ between $0$ and $k-1$.

A CDGA $A$ is called a \emph{Sullivan} algebra, if there is a generating set
of homogeneous elements $a_i\in A$ indexed by a well  ordered set $I$, such
that
\begin{enumerate}[(i)]
\item $da_k$ lies in the
subalgebra generated by the elements $a_j$ with $j< k$;
\item $A$  has a basis consisting of  the elements
\begin{equation*}
\begin{aligned}
a_{j_1}^{r_1} \dots a_{j_n}^{r_n},
\end{aligned}
\end{equation*}
with $j_1 < \dots < j_n$,
 $r_k=1$ if degree of $a_k$ is odd, and $r_k \in
\N$ if degree of $a_k$ is even.
\end{enumerate}
The motivating example of Sullivan CDGA was the Chevalley-Eilenberg
algebra $\bigwedge \liealg^*$ for nilpotent Lie algebra $\liealg$.
We will need the following lifting property of Sullivan CDGAs.
\begin{proposition}[{cf.  \cite[Proposition 12.9 ]{felixhalperin}}] \label{lifting}
Suppose $q\colon A\to B$ is a quasi-isomorphism of CDGAs and $f \colon D \to B
$ a homomorphism of CDGAs. If $D$ is Sullivan, then there is a homomorphism
$h\colon D \to A$ of CDGAs, such that $H^\bullet(q)\circ H^\bullet(h) =
H^\bullet (f)$.
\end{proposition}
An algebra $A$ quasi-isomorphic to $\Omega^\bullet(M)$ is called a \emph{real
homotopy model} of $M$. Suppose $M$ and $N$ are homotopy equivalent smooth
manifolds.  Let $F\colon M \to N$ and $G\colon N\to M$ be mutually inverse
homotopy equivalences. By   Whitney Approximation Theorem
(cf.~\cite[Thm.~6.26]{leebook}), there are smooth maps $\widetilde{F} \colon M \to N$ and
$\widetilde{G} \colon N \to M$ homotopy equivalent to $F$ and $G$,
respectively. Then $\widetilde{F}$ and $\widetilde{G}$ are mutually inverse
smooth homotopy equivalences. This implies that $\widetilde{F}^* \colon
\Omega^\bullet(N)\to \Omega^\bullet(M)$ and $\widetilde{G}^* \colon
\Omega^\bullet(M) \to \Omega^\bullet(N)$ are quasi-isomorphisms.
Hence the following proposition holds.
\begin{proposition}
\label{realhomotopymodels}
Suppose $M$ and $N$ are homotopy equivalent smooth manifolds. Then there are
quasi-isomorphisms $\Omega^\bullet(M) \to \Omega^\bullet(N)$ and
$\Omega^\bullet(N) \to \Omega^\bullet(M)$. In particular, $M$ and $N$ have the
same real homotopy models.
\end{proposition}

\subsection{Nilmanifolds}
Let $\Gamma$ be a torsion-free finitely generated nilpotent group.
It
was shown in~\cite{malcev,malcevtr}, that such $\Gamma$
can be realized as a lattice in a connected and simply connected nilpotent Lie
group $G(\Gamma)$ which is unique up to isomorphism.
We denote the canonical embedding of $\Gamma$ into $G(\Gamma)$ by
$\nu_\Gamma$.

We will use the following extension principle for lattices in nilpotent Lie
groups.
\begin{theorem}
\label{extension}
Let $H$ and $G$ be connected and simply connected nilpotent Lie groups. If $\Gamma$ is a lattice in $H$,
then every homomorphism of groups $f \colon \Gamma \to G$  has  a unique
extension to a smooth homomorphism $\tilde{f} \colon H \to G$.
\end{theorem}
\begin{proof}
This is essentially \cite[Theorem~2.11]{discrete}, but there the author claims
only the existence of continuous homomorphism $\tilde{f}\colon H \to G$ that
extends $f$. However, it
is known (cf. \cite[Theorem~3.39]{warner}) that  every continuous homomorphism of
Lie groups is smooth.
\end{proof}
 Denote  the \emph{nilmanifold} $\Gamma\backslash G(\Gamma)$ by $N_\Gamma$.
It is a manifold of type $K(\Gamma,1)$.
We will consider $N_\Gamma$ as a pointed manifold with the base point $\Gamma
e$.

\section{Approximation with nilmanifolds}
\label{topology}
The main result of this section is more general than needed for the proof
of \Cref{main2}.
For simplicity, the reader may assume that $M$ is an aspherical nilpotent manifold through the
section.
We decided to present the results in a more general way in order to be able to
compare  them to the construction in~\cite{chen} and with the hope that
\Cref{main1} can be useful in another context.

Let $(M,x_0)$ be a pointed manifold and $\Pi$ its fundamental group.
For every homomorphism $q\from \Pi \to \Gamma$ there is a unique up to homotopy
map $h_q \from M \to N_\Gamma$ of pointed manifolds such that $\pi_1(h_q)
= q$ (cf.~\cite[Prop.~1B.9]{hatcher}).
By   Whitney Approximation Theorem (cf.~\cite[Thm.~6.26]{leebook}), one can
always assume that $h_q$ is smooth, and any two such maps are connected by a
smooth homotopy.

The aim of this section is to give a parametrization of the homotopy class of smooth maps $h_q$
in the special case when $M$ is compact and $q$ is the canonical projection
onto $(\Pi/\Pi_k)_{tf}$ for a fixed natural number $k$.
Here $\Pi_k$ denotes as usual the $k$th member of lower central series of $\Pi$ and the
subscript \emph{tf} indicates the torsion-free part of a nilpotent group.
Denote $(\Pi/\Pi_k)_{tf}$ by $\Gamma$ and write $\liealg_\Gamma$ for the Lie
algebra of $G(\Gamma)$.
We will use $\liealg_\Gamma$-valued forms
on $M$ for the
parametrization.

We start by recalling some facts about Lie algebra-valued $1$-forms.
Let $G$ be a Lie group and $\liealg$ its Lie algebra.
We identify $\liealgdual$ with the left-invariant $1$-forms on $G$ and
$\extalg \liealgdual$ with the subcomplex $\derham\bullet(G)^G$ of all left-invariant
forms on $G$.

Let  $\omega \in \Omega^1(M,\liealg)$. The form
$\omega$ corresponds to a unique linear map $\tr{\omega} \from \liealg^*
\to \Omega^1(M)$, which extends to the unique homomorphism
\begin{align*}
\tr{\omega} \from \extalg \liealgdual \to \derham\bullet(M)
\end{align*}
of graded algebras.
If $h\from N \to M$ is a smooth map, then $h^* \omega\in \Omega^1(N,\liealg)$
corresponds to the linear map $h^* \circ \tr{\omega}$, i.e.\
\begin{equation}
\label{wthomega}
\begin{aligned}
\tr{h^*\omega} = h^* \circ \tr{\omega}.
\end{aligned}
\end{equation}
 The map $\tr{\omega}$ is a homomorphism of chain
complexes if and only if $\omega$ is \emph{flat}, i.e.\ if and only if it
satisfies the Maurer-Cartan equation
\begin{equation*}
\begin{aligned}
d\omega + \frac12 [\omega,\omega]=0.
\end{aligned}
\end{equation*}
Here $[\omega,\omega] \in\derham2(M,\liealg)$ is defined by $[\omega,\omega](X,Y) = 2
[\omega(X), \omega(Y)]$.

In particular, the inclusion $\extalg \liealgdual = \Omega(G)^G \to \derham\bullet(G)$
corresponds to a flat form $\mu_G \in \derham1(G,\liealg)$, called the
\emph{Maurer-Cartan form on $G$}.
We will later use that the form $\mu_{G(\Gamma)}$ is left invariant.

Given a smooth map $h \colon M\to G$, the homomorphism of CDGAs $h^*\circ
\tr{\mu}_G \from \extalg \liealgdual \to \derham\bullet(M)$
corresponds to the flat form $h^*\mu_G \in \derham1(M,\liealg)$.
The form $h^* \mu_G$ is called the \emph{Darboux derivative}
of $h$. The following result shows that in the case $M$ is simply connected,
every flat form is the Darboux derivative of a smooth map from $M$ to
$G$, unique up to translation.
\begin{theorem}[{\cite[Section 3.7]{sharpe}}]
\label{sharpe}
Let $(X,x_0)$ be a pointed simply connected manifold
 and $\omega$ a flat $\liealg$-valued
$1$-form on $X$.
For each $g\in G$ there
is a unique smooth map $P_{g,\omega}\colon X \to G $ such that
$P_{g,\omega}$ sends the base point of $X$ to $g$
 and $P^*_{g,\omega}\mu_G = \omega$. Moreover, $L_{g'} \circ P_{g,\omega} = P_{g'g,\omega}$ for any  $g'\in
G$.
\end{theorem}
We write $\widetilde{M}$ for the universal cover of $M$ and
$\pi$ for the projection from $\widetilde{M}$ to $M$.
We fix a point $\widetilde{x}_0\in \widetilde{M}$ over $x_0$ and consider
$\widetilde{M}$ as a pointed manifold with the base point $\widetilde{x_0}$.
If $\omega$ is a flat $\liealg$-valued $1$-form on $M$, then $\pi^*\omega$ is a
flat $1$-from in $\Omega^1(\widetilde{M},\liealg)$.
Then, $P^*_{e,\pi^*\omega} \mu_G = \pi^*\omega$ and \eqref{wthomega}  imply that
\begin{equation}
\label{picircomega}
\begin{aligned}
\pi^*\circ
\tr{\omega} = P^*_{e,\pi^*\omega}|_{\bigwedge \liealg^*}.
\end{aligned}
\end{equation}

\begin{corollary}
\label{rhoomega}
Let $\omega$ a flat $\liealg$-valued form on a manifold $M$.
There is a homomorphism of groups $\sigma_\omega \from \pi_1(M) \to G$, such that
for all $g \in \pi_1(M)$ and $x\in \widetilde{M}$
\begin{equation*}
\begin{aligned}
P_{e, \pi^*\omega} ( gx) = \sigma_\omega(g) P_{e,\pi^*\omega}(x).
\end{aligned}
\end{equation*}
\end{corollary}
\begin{proof}
Write $L_g$ for the operator on $\widetilde{M}$, that sends $x$ to $gx$. We
compute the Darboux derivative of $P_{e,\pi^*\omega} \circ L_g$. We get
\begin{equation*}
\begin{aligned}
(P_{e,\pi^*\omega}\circ L_g)^* \mu_G = L_g^* P^*_{e,\pi^*\omega} \mu_G = L_g^*
\pi^* \omega = (\pi\circ L_g)^* \omega = \pi^*\omega.
\end{aligned}
\end{equation*}
Denote the image of $\widetilde{x}_0$ under $P_{e,\pi^*\omega} \circ L_g$ by $\sigma_\omega(g)$.
By~\Cref{sharpe}, we get $P_{e,\pi^*\omega} \circ L_g = P_{\sigma_\omega(g),\pi^*\omega} = L_{\sigma_\omega(g)}
P_{e,\pi^*\omega}$. It is a routine to check that $\sigma_\omega$ is a homomorphism of
groups.
\end{proof}
Now we move to the special case when $G$ is the Lie group $G(\Gamma)$ with $\Gamma =
(\Pi/\Pi_k)_{tf}$.
As the pullback along the projection $G(\Gamma) \to N_\Gamma$ induces an
isomorphism between $\derham\bullet(N_\Gamma)$ and the set of $\Gamma$-invariant forms on
$G(\Gamma)$, we get an embedding of CDGAs $\psi_\Gamma \from \extalg
\liealg_\Gamma^* \to
\Omega^(N_\Gamma)$. It was shown in~\cite{nomizu} that
$\psi_\Gamma$ is a quasi-isomorphism. We denote by $\mu_\Gamma$ the corresponding
$\liealg_\Gamma$-valued $1$-form on $N_\Gamma$.

For every smooth map ${h} \from M \to N_\Gamma$ such that
$\pi_1(h)$ coincides with the canonical projection $q\from \Pi \to
\Gamma$, we get a $\liealg_\Gamma$-valued
$1$-form $h^* \mu_\Gamma$ on~$M$. The corresponding homomorphism of CDGAs is
$h^*\circ \psi_\Gamma$.
The induced map $H^1(h^*\circ \psi_\Gamma)\from
H^1(\liealg_\Gamma)\to H^1(M,\R)$ is an isomorphism.

\begin{proposition}\label{postnikovfirstfloor}
If a flat $\liealg_\Gamma$-valued $1$-form $\omega$ on $M$ is such that
$H^1(\tr{\omega}) \from H^1(\liealg_\Gamma)\to H^1(M)$ is
an isomorphism, then there is a unique automorphism $A_\omega \from G(\Gamma) \to
G(\Gamma)$ such that $ \sigma_\omega = A_\omega \circ \nu_\Gamma \circ q$.

\end{proposition}
\begin{proof}
The Lie group $G(\Gamma)$
is a nilpotent group with nilpotency class at most $k$.
Hence every map from $\Pi$ to $G(\Gamma)$ factors uniquely via the projection
from $\Pi$ onto $\Pi/\Pi_k$. Since $G(\Gamma)$ is torsion-free, every map from
$\Pi/\Pi_k$ to $G(\Gamma)$ admits a unique factorization via the projection
from $\Pi/\Pi_k$ onto $\Gamma = (\Pi/\Pi_k)_{tf}$.
Hence there is a unique homomorphism of groups $f\from \Gamma \to G(\Gamma) $ such that
$\sigma_\omega = f\circ q$.

By \Cref{extension}, there is a unique smooth
homomorphism of Lie groups $A_\omega \from G(\Gamma) \to G(\Gamma)$ that extends
$f$, i.e.\ $f = A_\omega|_{\Gamma} = A\circ  \nu_\Gamma$. Thus
$\sigma_\omega = A_\omega \circ \nu_\Gamma \circ q$.
The uniqueness of $A_\omega$ with such property follows from the uniqueness
part of~\Cref{extension} and the uniqueness of $f$ such that~$\sigma_\omega =
f\circ q$.

It is left to show that $A_\omega$ is an automorphism.
As $G(\Gamma)$ is a simply connected nilpotent Lie group, it is enough to check that
$A_\omega$ is onto.

Write $H$ for the image of $A_\omega$ in $G(\Gamma)$.
It is shown in
\cite[Theorem~7.18]{clement}, that for a nilpotent group $G$ and a subgroup  $N<G$, one has
$N[G,G]=G$ if and only if $N=G$.
Thus $A_\omega$ is onto if and only if  $H[G(\Gamma), G(\Gamma)] = G(\Gamma)$.

The group $H[G(\Gamma), G(\Gamma)]$ is a path-connected subgroup of
$G(\Gamma)$.
 By the result of Yamabe~\cite{yamabe} (see
also~\cite{goto}), every path-connected subgroup  of $G(\Gamma)$ is
a Lie subgroup. Thus it corresponds to a Lie subalgebra $\klie$ of
$\liealg_\Gamma$.
As $G(\Gamma)$ is a simply connected nilpotent Lie group, the exponential map
$\exp \from \liealg_\Gamma \to G(\Gamma)$
is a diffeomorphism. From the Baker-Campbell-Hausdorff formula and nilpotency of
$G(\Gamma)$, it follows that
$\exp(\klie)$ is a Lie subgroup of $G(\Gamma)$.
Hence $\exp(\klie)$ coincides with $H[G(\Gamma), G(\Gamma)]$.
As $\klie$ is a closed subset of $\liealg_\Gamma$ and $\exp\from \liealg_\Gamma
\to G(\Gamma)$ is a
diffeomorphism, we get that $H[G(\Gamma),G(\Gamma))]$ is a closed contractible subgroup of
$G(\Gamma)$.

As $H[G(\Gamma), G(\Gamma)]$ contains $[G(\Gamma), G(\Gamma)]$, it is a normal
subgroup of~$G(\Gamma)$.
Hence, the quotient  $Q := G(\Gamma)/ H[G(\Gamma), G(\Gamma))]$ is an  abelian
Lie group.
From the long exact sequence for the fibration $G\left( \Gamma \right)
\epi Q$,
we conclude that all the homotopy groups of $Q$ are trivial and thus $Q$ is
contractible.
Hence $Q \cong (\R^k, +)$ as a Lie group for some $k\ge 0$.
Clearly, $k=0$ if and only if $H[G(\Gamma),G(\Gamma)] = G(\Gamma)$, if and only
if $A_\omega$ is onto.

Suppose $k\not=0$. Then there is a non-trivial homomorphism $\phi\colon Q \to \R$ of Lie groups.
Denote by $\pr_Q$ the canonical projection of $G(\Gamma)$ onto $Q$.
Write $\beta $ for $(\phi \circ \pr_Q)^*(dt)$.
Since $dt$ is a non-zero left-invariant form on $(\R,+)$ and $\phi\circ \pr_Q$ is a
surjective homomorphism of Lie algebras, the form $\beta$ is a non-zero
left-invariant form on $G(\Gamma)$.
Hence $\beta \in \liealg_\Gamma^*$. As $\beta$ is closed, we get that $\beta
\in \liealg_\Gamma^*\cap \ker d_{CE} = H^1(\bigwedge \liealg_\Gamma^*)$.

As $H^1(\tr{\omega})$ is injective,
the class
$[\tr{\omega}(\beta)]$ is a non-zero element in $H^1(M)$.
Hence
there is a loop $\gamma$ at $x_0$ such that $\int_{\gamma}
\tr{\omega} (\beta) \not=0$.
Let $\widetilde{\gamma} \from \R \to \widetilde{M}$ be the unique lifting of $\gamma$
such that $\widetilde{\gamma}(0)=\tilde{x}_0$.
 Using \eqref{picircomega}, we get
\begin{equation*}
\begin{aligned}
0 & \not= \int_{\gamma} \tr{\omega}(\beta) = \int_0^1 \gamma^*
(\tr{\omega} (\beta))
= \int_0^1 \widetilde{\gamma}^*\circ \pi^* \circ \tr{\omega} (\beta)\\&
= \int_0^1 \widetilde{\gamma}^*\circ  P_{e,\pi^*\omega}^* \circ
 \pr_Q^* \circ \phi^*(dt) = \int_c
\phi^*(dt),
\end{aligned}
\end{equation*}
where $c$ is the curve $\pr_Q \circ P_{e,\pi^*\omega} \circ
\widetilde{\gamma}$ in $Q$.
Consider the path $P_{e,\pi^*\omega}\circ \widetilde{\gamma}$ in $G(\Gamma)$. We
have $\widetilde{\gamma}(1) = [\gamma]{x}_0$ by definition of the
action of $\Pi$ on $\widetilde{M}$.
By~\Cref{rhoomega}, we get
\begin{equation*}
\begin{aligned}
P_{e,\pi^*\omega}(\widetilde{\gamma}(1)) =
P_{e,\pi^*\omega}([\gamma]x_0) = \sigma_\omega([\gamma]) =A_\omega \circ
\nu_\Gamma \circ q ([\gamma]).
\end{aligned}
\end{equation*}
Thus $P_{e,\pi^*\omega} (\widetilde{\gamma}(1))$ lies in the image of
$A_\omega$. Thus $c(1)$ is the neutral element of $Q$. Obviously $c(0)$ is also
the neutral element in $Q$. Therefore $c$ is a loop in
a contractible space.  We get $\int_{c} \phi^* dt =0$.
This gives a contradiction to the assumption $k\not=0$. Therefore $k=0$ and
$A_\omega$ is an automorphism.
\forus{We used only that $\tr{\omega}\from \extalg
\liealgdual \to \derham\bullet(M)$ induces an injective map on
$H^1$. But $\dim H^1(\extalg \liealgdual) = \dim (\liealg/[\liealg,\liealg]) =
\dim G/[G,G] = \rank \Gamma/[\Gamma,\Gamma] = \dim H_1(M) = \dim H^1(M)$.}
\end{proof}
For a flat $1$-form $\omega\in \Omega^1(M,\liealg_\Gamma)$ such that
$H^1(\omega)$ is an isomorphism, denote by $\widetilde{h}_\omega$ the smooth map
$A_\omega^{-1} \circ P_{e,\pi^*\omega}$ from $\widetilde{M}$  to $G(\Gamma)$.
For every $g\in \Pi$ and $y \in \widetilde{M}$, we get $\widetilde{h}_\omega(gx) =
q(g) \widetilde{h}_\omega(x)$.
Moreover, $\widetilde{h}_\omega(\widetilde{x}_0) =e$.
Therefore, $\widetilde{h}_\omega$ induces a smooth map $h_\omega \from M \to
N_\Gamma$ of pointed manifolds such that $\pi_1(h_\omega) = q$.
Write $\pr_\Gamma$ for the canonical projection $G(\Gamma) \to N_\Gamma$.
We get
\begin{equation}
\label{pistarrho}
\begin{aligned}
\pi^* \circ \tr{\omega} = P_{e,\pi^*\omega}^*|_{\bigwedge \liealg_\Gamma^*} =
\tilde{h}\strut^*_\omega \circ
A_\omega^*|_{\bigwedge \liealg_\Gamma^*} = \tilde{h}\strut^*_\omega|_{\bigwedge
\liealg_\Gamma^*} \circ a_\omega,
\end{aligned}
\end{equation}
where $a_\omega$ is the restriction of $A^*_\omega\from \Omega(G(\Gamma)) \to
\Omega(G(\Gamma))$ to $\bigwedge \liealg_\Gamma^*$.
The map  $\pr_\Gamma^* \circ \psi_\Gamma$
coincides with the canonical inclusion of $\bigwedge \liealg_\Gamma^*$ into $\Omega^\bullet(G(\Gamma))$.
As $h_\omega \circ \pi = \pr_\Gamma \circ \widetilde{h}_\omega$, we get
\begin{equation*}
\begin{aligned}
\tilde{h}\strut^*_\omega|_{\bigwedge \liealg_\Gamma^*} = \tilde{h}\strut^*_\omega \circ \pr_\Gamma^* \circ
\psi_\Gamma = \pi^* \circ h_\omega^* \circ \psi_\Gamma.
\end{aligned}
\end{equation*}
Combining this equation with~\eqref{pistarrho}, we obtain
\begin{equation*}
\begin{aligned}
\pi^* \circ \tr{\omega} = \pi^* \circ h_\omega^* \circ \psi_\Gamma \circ
a_\omega.
\end{aligned}
\end{equation*}
Since $\pi^*\from \Omega(M) \to \Omega(\widetilde{M})$ is injective, we get
$\tr{\omega} = h_\omega^* \circ \psi_\Gamma \circ a_\omega$.
Hence $\tr{\omega} \circ a_\omega^{-1}= h_\omega^* \circ \psi_\Gamma$.
\forus{The map on the right side corresponds to the $1$-form $h^*\mu_\Gamma$.
\begin{equation*}
\begin{aligned}
\Hom (\liealg^*, \Omega^1(M)) & \to \Omega^1(M, \liealg^{**}) \\
f & \mapsto ( X \mapsto ( \alpha \mapsto f(\alpha)(X))) \\
\tr{\omega} \circ a_\omega^{-1} & \mapsto ( X \mapsto (\alpha \mapsto
\tr{\omega} ( a_\omega^{-1} (\alpha) ) (X) = a_\omega^{-1} (\alpha)
( \omega(X)) )
\end{aligned}
\end{equation*}
\begin{equation*}
\begin{aligned}
\tr{\omega}(\beta) (X) = \beta (\omega(X))
\end{aligned}
\end{equation*}
We have to identify the element $y$ of $\liealg$ that corresponds to the element
$\alpha \mapsto a_\omega^{-1}(\alpha) (\omega(X))$ of $\liealg^{**}$, i.e.\
for all $\alpha\in \liealg^*$, we should have $a_\omega^{-1}(\alpha)
(\omega(X)) = \alpha(y)$.
Thence $y = (a_\omega^{-1})^t \omega(X)$.
Therefore $h_\omega^* \mu_\Gamma = (a_\omega^{-1})^t \circ \omega$.}
Therefore, we proved the existence part of the following theorem.
\begin{theorem}\label{main1}
Let $(M,x_0)$ be a pointed compact manifold and $k$ a natural number.
Write $\Pi$ for the fundamental group of $M$, $\Gamma$ for
$(\Pi/\Pi_k)_{tf}$, and $q$ for the canonical projection from $\Pi$ to
$\Gamma$.
For every flat $1$-form $\omega \in \Omega^1(M,\liealg_\Gamma)$  such that
$H^1(\tr{\omega})$ is an isomorphism, there is a
unique smooth map $h_\omega \from M \to N_\Gamma$ of pointed manifolds  and a unique automorphism
$a_\omega$ of $\bigwedge\liealg_\Gamma^*$ such that $\pi_1(h_\omega)=q$ and
$\tr{\omega} = h_\omega^* \circ \psi_\Gamma \circ a_\omega$.
\end{theorem}
\begin{proof}[Proof of the uniqueness]
Suppose $h\from M \to N_\Gamma$
and $b \in \Aut(\bigwedge \liealg_\Gamma^*)$ are such that $h(x_0) = \Gamma e$,
$\tr{\omega} = h^* \circ \psi_\Gamma \circ b$, and $\pi_1(h) = q$.
There is a unique lifting $\widetilde{h}\from
\uc{M} \to G(\Gamma)$
of $h$ such that $\widetilde{h}(\widetilde{x}_0) = e$.
Denote by $B$ the automorphism of $G(\Gamma)$ that integrates
$b^t|_{\liealg_\Gamma} \from \liealg_\Gamma \to \liealg_\Gamma$.
We claim that $B\circ \widetilde{h}$ and $P_{e,\pi^*\omega}$
are equal. As both maps send $\widetilde{x}_0$ to $e$, by the uniqueness part
of~\Cref{sharpe}, it is enough to show that they have the same Darboux
derivative.  The $1$-form $(B\circ \widetilde{h})^*\mu_{G(\Gamma)}$ corresponds
to the homomorphism $\widetilde{h}^*\circ B^*|_{\bigwedge \liealg_\Gamma^*}
\from \bigwedge \liealg_\Gamma^* \to \Omega^*(\uc{M})$ of CDGAs. We have
\begin{equation*}
\begin{aligned}
 \widetilde{h}^*\circ B^*|_{\bigwedge
\liealg_{\Gamma}^*} = \widetilde{h}^* \circ \pr_\Gamma^* \circ \psi_\Gamma \circ b =
\pi^*\circ h^* \circ \psi_\Gamma \circ b = \pi^* \circ \tr{\omega} =
\tr{\pi^* \omega}.
\end{aligned}
\end{equation*}
Hence $(B\circ \widetilde{h})^* \mu_{G(\Gamma)} = \pi^*\omega =
P^*_{e,\pi^*\omega} \mu_{G(\Gamma)}$ and $B\circ \widetilde{h} =
P_{e,\pi^*\omega}$.
In particular, $\sigma_\omega(g)  = B (\widetilde{h} (g\widetilde{x}_0))$. Since
$\pi_1(h) = q$, we have $\widetilde{h} (gx) = q(g) \widetilde{h}(x)$. Thus
$\sigma_\omega(g) =B(q(g))$. The uniqueness part of~\Cref{postnikovfirstfloor}
implies that $B=A_\omega$ and, hence, $b=a_\omega$. Now, $B\circ \widetilde{h} = P_{e,\pi^*\omega} =
A_\omega \circ \widetilde{h}_\omega$ implies that $\widetilde{h} =
\widetilde{h}_\omega$. Hence $h = h_\omega$.
\end{proof}

\begin{remark}
In~\cite{chen}, Chen constructed a smooth
homotopy equivalence $h\from M \to N_\Gamma$ such that $\pi_1(h) =q$ starting with a splitting of
$\Omega^\bullet(M)$. In particular, his construction can be applied to any
Riemannian manifold. Our understanding is that the corresponding
$\liealg_\Gamma$-valued $1$-form can be obtained from the Chen connection in
$\Omega^\bullet (M) \otimes \widehat{\mathbb{T}}(H_\bullet(M))$.
We do not know if every smooth homotopy equivalence $h\from M\to N_\Gamma$ such that
$\pi_1(h) =q$ corresponds to a suitable splitting of $\Omega^\bullet(M)$.
\end{remark}

\begin{remark}
\label{hsurjective}
Suppose $M$ aspherical and
$\Pi = \Gamma$, i.e.\ that $\Pi$ is a torsion-free nilpotent group of
nilpotency class at most $k$.
If a flat $\liealg_\Gamma$-valued $1$-form $\omega$ on $M$ is such that
$\tr{\omega}$ is a quasi-isomorphism, then $h_\omega$ is a homotopy
equivalence. In particular, the degree of $h_\omega$ is either $1$ or $-1$.
As a non-zero degree map between compact manifolds, $h_\omega$ is surjective (cf. \cite[Prop. I, Sec. 6.1]{greub1}).
\end{remark}
\begin{corollary}
\label{rhoinjective}
Let $M^n$ be a compact aspherical manifold with nilpotent fundamental group
$\Gamma$ and $\liealg$ the Lie algebra of $G(\Gamma)$. If $\rho\colon \bigwedge \liealg^* \to \Omega^\bullet(M)$
is a quasi-isomorphism of CDGAs then $\rho$ is an injective map.
\end{corollary}
\begin{proof}
Let $\omega\in \Omega^1(M,\liealg)$ be such that $\rho = \tr{\omega}$.
By \Cref{main1}, $\rho = h_\omega^* \circ \psi_\Gamma \circ a_\omega$, where
$a_\omega$ is an automorphism of $\bigwedge\liealg_\Gamma$.
 It is  clear that $\psi_\Gamma \colon \bigwedge \liealg^* \to
\Omega^*(N_\Gamma)$ is injective. The map $h_\omega$ is surjective by the previous
remark. Thus $h_\omega^*\colon \Omega^\bullet(N_\Gamma) \to \Omega^\bullet(M)$ is injective.
Hence $\rho$ is a composition of three injective maps.
\end{proof}

\section{ On de Rham algebra of  Sasakian manifolds}
\label{sasakianmanifolds}
The main objective of this section is to show that for every compact aspherical
Sasakian manifold with nilpotent fundamental group $\Gamma$, there is a
quasi-isomorphism $\bigwedge T^*_e G(\Gamma) \to \Omega^\bullet(M)$ with
sufficiently rigid  properties that will  imply
that a corresponding smooth homotopy equivalence $h\colon M \to N_\Gamma$ is a
diffeomorphism.

We will use the following notation.
Given  linear operators $A_1$, \dots, $A_k$ on $\Omega^\bullet(M)$, we will
write $\Omega^\bullet_{A_1,\dots, A_k}(M)$ for the intersection of the kernels
of the operators $A_1$, \dots, $A_k$.
If each of $A_1$, \dots, $A_k$ is a homogeneous graded derivation, then
$\Omega^\bullet_{A_1,\dots,A_k}(M)$ is a subalgebra of~$\Omega^\bullet(M)$.
If $R$ is a graded subalgebra of $\derham\bullet(M)$ and $\alpha$ is an element of
$\derham\bullet(M)$, then $R[\alpha]$ denotes the subalgebra of
$\derham\bullet(M)$ generated by $R$ and $\alpha$. If $\alpha$ has an odd
degree, then, of course, $R[\alpha] = R + R\wedge \alpha$.

Now we discuss Tievsky models for the complexified de Rham algebra of a Sasakian manifold constructed in his PhD
thesis~\cite{tievsky}.
Let $(M^{2n+1},\varphi,\xi,\eta)$ be a Sasakian manifold.
Tievsky showed that the embedding
\begin{equation*}
\begin{aligned}
\Omega^\bullet_{i_\xi,\lie_\xi,d_B^c}(M)_\C [\eta] \hookrightarrow \Omega^\bullet(M)_\C
\end{aligned}
\end{equation*}
is a quasi-isomorphism. Since the base field change functor is exact, we get
that also
\begin{equation}
\label{tievskyquasiiso}
\begin{aligned}
\Omega^\bullet_{i_\xi,\lie_\xi,d_B^c}(M) [\eta] \hookrightarrow
\Omega^\bullet(M)
\end{aligned}
\end{equation}
is a quasi-isomorphism.
We write $\tievsky^\bullet(M)$ for
the CDGA
$\Omega^\bullet_{i_\xi,\lie_\xi,d_B^c}(M) [\eta]$  and refer to it as the
\emph{first Tievsky model}.

By~\cite[Lemma~1]{wolak}, the $d_B d_B^c$-lemma holds on
$\Omega^\bullet_{i_\xi,\lie_\xi}(M)$, i.e.\
\begin{equation*}
\begin{aligned}
\ker d_B^c \cap \ker d_B \cap \im d_B = \im ( d_B d_B^c).
\end{aligned}
\end{equation*}
This implies that
\begin{equation}
\label{kerdbc}
\begin{aligned}
\ker d_B^c = \ker d_B^c \cap \ker d_B + \im d_B^c.
\end{aligned}
\end{equation}
The inclusion ``$\supset$'' is obvious. Applying $d_Bd_B^c$-lemma to $d_B \beta$
with $\beta \in \ker d_B^c$, we get that there is $\alpha$ such that $d_B \beta
= d_B d_B^c \alpha$. Thus
\begin{equation*}
\begin{aligned}
\beta = ( \beta - d_B^c \alpha ) + d_B^c \alpha \in \ker d_B^c \cap d_B + \im
d_B^c.
\end{aligned}
\end{equation*}
This permits to identify the first two components of $\tievsky^\bullet(M)$.
\begin{proposition}
\label{Omega1liexiliephi}
Let $M^{2n+1}$ be a compact Sasakian manifold. Then
\begin{align}
\label{Omega0liexilievarphi}
\tievsky^0(M) & = \Omega^0_{i_\xi,\lie_\xi, d_B^c}(M) =  \R\\
\label{omegafirstixiliexidbc}
\tievsky^1(M) & = \Omega^1_{\lap}(M) \oplus d_B^c
\Omega^0_{i_\xi,\lie_\xi}(M) \oplus \R\eta.
\end{align}
 Moreover, $H^1(\tievsky^\bullet(M))$ coincides with $\Omega^1_\lap(M)$.
\end{proposition}
\begin{proof}
Applying \eqref{kerdbc} in degree zero, we get
\begin{equation*}
\begin{aligned}
\R \subset \Omega^0_{i_\xi,\lie_\xi,d_B^c}(M) =
\Omega^0_{i_\xi,\lie_\xi,d_B^c,d_B}(M) \subset \Omega^0_{d}(M) = \R.
\end{aligned}
\end{equation*}
This shows~\eqref{Omega0liexilievarphi}.
In degree $1$, we get
\begin{equation*}
\begin{aligned}
\Omega^1_{i_\xi,\lie_\xi,d_B^c} (M) =
\Omega^1_{i_\xi,\lie_\xi,d_B^c,d_B} (M)  + d_B^c \Omega^0_{i_\xi,\lie_\xi}(M).
\end{aligned}
\end{equation*}
The sum on the right side is direct. Indeed,
suppose $\beta =
d_B^c h \in \Omega^1_{i_\xi,\lie_\xi,d_B^c,d_B}(M)$ for some basic function
$h$. Then by $d_Bd_B^c$-lemma $\beta$ is in the image of $d_B d_B^c$, and, for the degree
reasons, $\beta=0$.
Thus
\begin{equation}
\label{firstcomponentsecondtievsky}
\begin{aligned}
\tievsky^1(M) = \Omega^1_{i_\xi,\lie_\xi,d_B^c,d_B} (M)  \oplus d_B^c
\Omega^0_{i_\xi,\lie_\xi}(M) \oplus \R\eta.
\end{aligned}
\end{equation}

As the zeroth component of the model is $\R$, we see that the image of the
differential in the first component is zero. Hence the first cohomology group
of the Tievsky model coincides with the kernel of $d_B$ restricted to
$\tievsky^1(M)$.
By~\cite{tachibana} and~\cite[Thm.~4.1]{fujitani}, it is known that every harmonic $1$-form on a Sasakian manifold is
basic and that the space $\Omega^1_\lap(M)$ is invariant under the action of
$i_\varphi$. The second property implies that for every harmonic $1$-form
$\alpha$ we have $d_B^c \alpha = -\lie_\varphi \alpha = d (i_\varphi \alpha) =
0$. Thus
\begin{equation*}
\begin{aligned}
\Omega^1_\lap (M) \subset \Omega^1_{i_\xi,\lie_\xi,d_B^c,d_B}(M) \subset
H^1(\Omega^\bullet_{i_\xi,\lie_\xi,d_B^c}(M)[\eta] ).
\end{aligned}
\end{equation*}
For dimension reasons the above inclusions are equalities.
This shows the last claim of the proposition.
Finally,~\eqref{omegafirstixiliexidbc} follows from
\eqref{firstcomponentsecondtievsky} and
$\Omega^1_\lap(M) = \Omega^1_{i_\xi,\lie_\xi,d_B^c,d_B} (M)$.
\end{proof}

Now we introduce the second Tievsky model.
Consider the graded algebra $H_B(M)[t]/t^2$ with \mbox{$\deg(t) =
1$}.
It is a CDGA with the differential $d(a + bt) = b [d\eta]_B$.
Tievsky showed that the epimorphism
\begin{equation*}
\begin{aligned}
\tievsky^\bullet(M)_\C = \Omega^\bullet_{i_\xi,\lie_\xi,d_B^c}(M)_\C[\eta]  & \epi H_B(M)_\C [t]/t^2\\
\alpha + \beta \wedge \eta & \mapsto [\alpha] + [\beta] t
\end{aligned}
\end{equation*}
is a quasi-isomorphism. This implies that the corresponding map
\begin{equation*}
\begin{aligned}
\tievsky^\bullet(M)  \epi H_B(M) [t]/t^2
\end{aligned}
\end{equation*}
is a quasi-isomorphism. Thence $H_B(M)[t]/t^2$ is a model of
$\Omega^\bullet(M)$. We call it the \emph{second Tievsky model} of $M$.

It was observed by Bazzoni in~\cite{bazzoni} that our proof of the main result
in~\cite{imrn} can be modified to imply the following proposition.
For completeness, we give here a proof based on our results
in~\cite{israel}.
\begin{proposition}
\label{heisenberg}
If $M^{2n+1}$ is a compact aspherical Sasakian manifold with nilpotent
fundamental group $\Gamma$, then $G(\Gamma)$ is isomorphic to the Heisenberg
group  $H(1,n)$.
 \end{proposition}
\begin{proof}
The nilmanifold $N_\Gamma$ and $M$ are both of type $K(\Gamma,1)$, and, hence,
they are homotopy equivalent to each other.
By \Cref{realhomotopymodels},
the second Tievsky model of $M$ is  a real homotopy model of  $N_\Gamma$.
It was shown in~\cite[Thm.~5.3]{israel} that if  a nilmanifold $N_\Gamma$ of
dimension $2n+1$ has a real
homotopy model of the form $A[t]/t^2$ with $\deg(t)=1$, $dt\in A$, $[dt]^n
\not=0$, and the zero
differential on $A$, then $G(\Gamma)$ is isomorphic to $H(1,n)$.
As $[d\eta]_B^n \not=0$ by \cite[Lemma~3.1]{imrn}, we get the result.
\end{proof}
The following result is a corollary of the theory of Sullivan models.
\begin{proposition}
\label{imquasiiso}
Let $M^{2n+1}$ be a compact aspherical Sasakian manifold with nilpotent
fundamental group $\Gamma$. Then
there is a quasi-isomorphism of CDGAs $\rho\colon \bigwedge \liealg_\Gamma^* \to
\Omega^\bullet(M) $ such that $\im\rho \subset \tievsky^\bullet(M)$.
 \end{proposition}
\begin{proof}
By \Cref{realhomotopymodels}, there is a quasi-isomorphism $f\from
\Omega^\bullet(N_\Gamma) \to \Omega^\bullet(M)$.
Write $i$ for the embedding of $\tievsky^\bullet(M)$ into
$\Omega^\bullet(M)$.
Since $\liealg_\Gamma$ is
nilpotent, the CDGA $\bigwedge \liealg_\Gamma^*$ is Sullivan. By the lifting property
for Sullivan algebras (see Prop.~\ref{lifting}) applied to  $f\circ \psi_\Gamma
\from \bigwedge \liealg_\Gamma^* \to \Omega^\bullet(M)$ and $i\from
\tievsky^\bullet(M) \to \Omega^\bullet(M)$, there is a homomorphism
of CDGAs
\begin{equation*}
\begin{aligned}
j\colon \bigwedge \liealg_\Gamma^* \to \tievsky^\bullet(M)
\end{aligned}
\end{equation*}
such
that $H^\bullet(i) \circ H^\bullet(j) = H^\bullet(f\circ \psi_\Gamma)$.
Since $i$ and $f\circ \psi_\Gamma$ are quasi-isomorphisms, we get that also
$j$ is a quasi-isomorphism. Thus $\rho := i\circ j \colon \bigwedge
\liealg_\Gamma^* \to \Omega^\bullet(M)$ is a quasi-isomorphism of CDGAs with the
claimed
property.
\end{proof}

\begin{proposition}
\label{etaheis}
Let $M^{2n+1}$ be a compact aspherical Sasakian manifold with nilpotent
fundamental group $\Gamma$.
If
  $\rho\colon \bigwedge \liealg_\Gamma^* \to
\Omega^\bullet(M) $ a quasi-isomorphism such that $\im\rho \subset \tievsky^\bullet(M)$, then $\rho$
induces an isomorphism between the space  of closed elements
in
$\bigwedge^1 \liealg_\Gamma^*$  and
$\Omega^1_\lap(M)$.
Moreover, there is $\tileta\in
\liealg_\Gamma^*$ and $f\in \Omega^0_{\lie_\xi}(M)$ such that $\rho(\tileta) = \eta +
\lie_\varphi f$.
\end{proposition}
\begin{proof}
By \Cref{Omega1liexiliephi}, we know that $\Omega^1_\lap(M) = H^1(\tievsky^\bullet(M))$.
Write $Z^1$ for the set of closed elements in $\bigwedge^1 \liealg_\Gamma^* $.
As $\bigwedge^0 \liealg_\Gamma^* =\R$, we have $Z^1 = H^1(\bigwedge \liealg_\Gamma^* )$.

Denote by $\hat\rho\colon \bigwedge \liealg_\Gamma^*  \to \tievsky^\bullet(M)$ the
map $\omega \mapsto \rho(\omega)$.
As $\rho$ and the inclusion  $\tievsky^\bullet(M) \hookrightarrow
\Omega^\bullet(M)$  are quasi-isomorphisms, also
 $\hat\rho$ is a quasi-isomorphism.
Hence the
quasi-isomorphism $\hat\rho$ induces an  isomorphism between $ Z^1$ and  $\Omega^1_\lap(M)$.

Now we will show the existence of $\tileta \in \liealg_\Gamma^* $ and $f\in
\Omega^0_{\lie_\xi}(M)$ with the
claimed properties.
By \Cref{heisenberg} the Lie algebra $\liealg_\Gamma$ is isomorphic to the
Heisenberg Lie algebra.
In particular, the codimension of $Z^1$ in $\liealg_\Gamma^* $ is one.
Choose an arbitrary $\beta \in
\liealg_\Gamma^*  \setminus Z^1$.
From \Cref{Omega1liexiliephi}, we get that  $\rho(\beta) =\hat\rho(\beta) = \omega +
\lie_\varphi h + \lambda\eta $ for some  $\omega \in \Omega^1_\lap(M)$, $h \in
\Omega^0_{\lie_\xi}(M)$ and $\lambda \in \R$. As $\rho|_{Z^1}\colon Z^1 \to
\Omega^1_{\lap}(M)$ is an isomorphism, there is $\alpha \in Z^1$ such that
$\rho (\alpha) =\omega$. Notice that $\beta-\alpha \not \in Z^1$. Hence,
replacing $\beta$ with $\beta -\alpha$, we can assume that $\omega =0$.

Next, we show that $\lambda\not=0$.
The top component of $\hat\rho$ can be seen as a map from $\R\beta \wedge
(\bigwedge^{2n} Z^1)$ to $\tievsky^{2n+1}(M) =
\Omega^{2n}_{B}(M,\foliation)\wedge \eta$.
If $\lambda=0$, then $\hat\rho(\R \beta) \subset \R (d_B^c h) \subset \Omega^1_{B}(M,\foliation)$.
As $\hat\rho(Z^1)$ is also a subset of $\Omega^1_{B}(M,\foliation)$, we get
that $\hat\rho(\R\beta \wedge ( \bigwedge^{2n} Z^1))$ should lie in the zero
space $\Omega^{2n+1}_B(M,\foliation)$. But then $H^{2n+1}(\hat\rho)$ is a zero
map, which contradicts the assumption that $\hat\rho$ is a
quasi-isomorphism.
 This shows that $\lambda\not=0$.

Now define $\tileta = (1/\lambda) \beta$.
We get $\rho(\tileta) = \lie_\varphi f + \eta$, where $f = (1/\lambda) h$.
\end{proof}

\section{Proof of \Cref{main2}}
\label{PROOF}

Let $(M^{2n+1},\varphi,\xi,\eta)$ be a compact nilpotent aspherical Sasakian manifold with
nilpotent group $\Gamma$.
By \Cref{heisenberg}, the group $G(\Gamma)$ is isomorphic to the Heisenberg
group $H(1,n)$, and, hence, $M$ is homotopy equivalent to a Heisenberg
nilmanifold $N_\Gamma = \Gamma \bs G(\Gamma)$.

\Cref{imquasiiso} implies that there is a $\liealg_\Gamma$-valued $1$-form
$\omega$ on $M$ such that $\tr{\omega} \from \bigwedge {\liealgdual}_\Gamma \to
\Omega^\bullet(M)$ is a quasi-isomorphism of CDGAs whose image lies in
$\tievsky^\bullet(M)$.
By \Cref{main1}, $\tr{\omega} = h^*_\omega \circ \psi_\Gamma \circ a_\omega$.
Write $h$ for $h_\omega$ and $\rho$ for $\tr{\omega} \circ a_\omega^{-1} = h^*
\circ \psi_\Gamma$.
By \Cref{hsurjective}, the smooth map $h\from M \to N_\Gamma$ is a homotopy
equivalence, and, therefore, it is surjective and has degree $\pm 1$.

The image of $\rho = \tr{\omega} \circ a^{-1}_\omega$ lies in $\tievsky^\bullet(M)$. Hence, by \Cref{etaheis}, there is
$\tileta\in{\liealgdual_\Gamma}$ and $f\in \Omega^0_{\lie_\xi}(M)$, such that $h^*
\tileta = \rho (\tileta) = \eta + \lie_\varphi f$.

Our next aim is to show that $h$ is $df$-twisted with respect to a
suitable left-invariant normal almost contact structure on $N_\Gamma$.

Let $\eta_N = \psi_\Gamma (\tileta) \in \derham1(N_\Gamma)$.
To define $\xi_N$, we have to ensure that $\xi$ is projectable.

For every point $x$ in $N_\Gamma$ the map
\begin{equation*}
\begin{aligned}
\psi_{\Gamma,x} \from {\liealgdual_\Gamma}&  \to T^*_x N_\Gamma \\
\alpha & \mapsto \psi_\Gamma(\alpha)_x
\end{aligned}
\end{equation*}
is an isomorphism of vector spaces.
Write $Z^1$ for the set of closed elements in ${\liealgdual_\Gamma}$. As $Z^1 \oplus
\left\langle \tileta \right\rangle = {\liealgdual_\Gamma}$, we get that $T^*_x
N_\Gamma = \psi_{\Gamma,x}(Z^1) \oplus \left\langle \eta_{N,x} \right\rangle$
for all $x \in N_\Gamma$.
\begin{claim}
\label{projectable}
\begin{enumerate}[(i) ]
\item \label{etaNhxix} For all $x \in M$, we have $\eta_N ( Th (\xi_x) ) =1$.
\item \label{psigammaz1}
For each $x\in M$, we have $\psi_{\Gamma,h(x)} (Z^1) = \Ann (Th (\xi_x))$.
\item \label{item:proj} The vector field $\xi$ is $h$-projectable.
\end{enumerate}
\end{claim}
\begin{proof}
 For every
$x\in M$, we have
\begin{equation*}
\begin{aligned}
\eta_N (Th(\xi_x)) & =  h^*(
\psi_\Gamma (\tileta)) ( \xi_x) = \rho(\tileta) (\xi_x)  =  (\eta  +
i_\varphi df ) (\xi_x) = 1.
\end{aligned}
\end{equation*}
For every $\alpha \in Z^1$, we get
\begin{equation*}
\begin{aligned}
\psi_{\Gamma}(\alpha) (Th(\xi_{x})) & = (h^*\circ \psi_\Gamma)
(\alpha) (\xi_x) = \rho(\alpha) (\xi_x) = (i_\xi \rho(\alpha))_x.
\end{aligned}
\end{equation*}
By~\Cref{etaheis}, the $1$-form $\rho(\alpha)$ is harmonic, and, hence, basic.
\forus{Check the current formulation of \Cref{etaheis}!}
Thus $\psi_\Gamma(\alpha)(Th( \xi_x )) =0$.  This shows that
$\psi_{\Gamma,x}(Z_1) \subset \Ann (Th (\xi_{x}))$ for all $x\in M$.
It follows from~\eqref{etaNhxix} that $\dim \Ann (Th(\xi_{x})) = 2n = \dim Z^1$.
Thus $\psi_{\Gamma,x}(Z_1) = \Ann (Th (\xi_{x}))$.

It is left to show that $\xi$ is projectable.
We already know that $h$ is surjective.
Suppose $x$, $x'\in M$ are in the preimage of $y\in N_\Gamma$. By \eqref{psigammaz1},
we have
\begin{equation*}
\begin{aligned}
\Ann( Th( \xi_x)) = \psi_{\Gamma, y} (Z^1) = \Ann (Th (\xi_{x'})).
\end{aligned}
\end{equation*}
Hence $Th( \xi_x)$ and $Th(\xi_{x'})$ are collinear. By~\eqref{etaNhxix}, we know that
$\eta_N(Th( \xi_x)) = 1 = \eta_N(Th(\xi_{x'}))$. Therefore $Th( \xi_x) = Th
(\xi_{x'})$.
\end{proof}
For $y\in N_\Gamma$, define
$\xi_{N,y} = Th(\xi_x)$, where $x$ is any point in the preimage of $y$.
As $\xi$ is projectable, the vector field $\xi_N$ is well defined.

Instead of defining $\varphi_N$ on $TN_\Gamma$, we define its transpose $\varphi_N^t$ on
$T^*N_\Gamma$.
By \Cref{etaheis}, the quasi-isomorphism $\rho$ induces an isomorphism $\tau \colon Z^1 \to
\Omega^1_\lap(M)$.
By~\cite[Thm.~4.1]{fujitani} the operator $i_\varphi$ preserves harmonic
forms.
For $y\in N_\Gamma$, define $(\varphi_N^t)_y \from T^*_y N_\Gamma \to T^*_y N_\Gamma$ to be zero on
$\left\langle \eta_{N,y} \right\rangle$ and to coincide with the
composition
\begin{equation*}
\begin{aligned}
\ker \xi_{N,y}  \xrightarrow{\psi_{\Gamma,y}^{-1} } Z^1 \xrightarrow{\tau}
\Omega^1_\lap(M) \xrightarrow{i_\varphi} \Omega^1_\lap(M)
\xrightarrow{\tau^{-1}} Z^1 \xrightarrow{\psi_{\Gamma,y}}
\ker \xi_{N,y} .
\end{aligned}
\end{equation*}
on $\ker (\xi_{N,y})$.
\begin{claim}
\label{normalalmostcontact}
The triple $(\varphi_N, \xi_N, \eta_N)$ is a left-invariant normal almost contact
structure on $N_\Gamma$.
\end{claim}
\begin{proof}
We already know that $\eta_N(\xi_N) =1$ by \Cref{projectable}\eqref{etaNhxix}.
It is immediate that $(\varphi^t_N)^2|_{\ker\xi_N} = -\id$ for all $x \in
N_\Gamma$. Also $(\varphi^t_N) \eta_N =0$.
Hence $(\varphi^t_N)^2 = -\id + \eta \otimes \xi$, i.e. $\varphi_N^2 = -\id +
\xi \otimes \eta$.
Hence $(\varphi_N, \xi_N, \eta_N)$ is an almost contact structure.

The form $\eta_N$ is left invariant by its definition.
\Cref{projectable}\eqref{psigammaz1} and the definition of $\xi_N$  imply that the space of smooth sections
of $\Ann(\xi_N) \subset T^*N_\Gamma$ is generated by left-invariant sections in $\psi_\Gamma(Z^1)$.
Hence $\Ann(\xi_N)$ is a left-invariant subbundle of $T^* N_\Gamma$. Thus
$\left\langle \xi_N \right\rangle$  is a left-invariant subbundle of
$TN_\Gamma$.
As $\eta_N$ and $\eta_N(\xi_N) =1$ are left-invariant, we conclude that
$\xi_N$ is a left-invariant vector field on $N_\Gamma$.
From the definition of $\varphi_N^t$ and the fact that $\Ann(\xi_N)$ and
$\left\langle \eta_N \right\rangle$ are
left-invariant subbundles of $T^*N_\Gamma$, it follows that $\varphi_N$ is left
invariant.

To check the normality condition $[\varphi_N,\varphi_N]_{FN} + 2 d\eta \otimes
\xi =0$, it is enough to show for every $x \in N_\Gamma$ and $\zeta \in T^*_x
N_\Gamma$, we have
\begin{equation*}
\begin{aligned}
\zeta \circ ([\varphi_N, \varphi_N]_{FN})_x
+ 2 \zeta(\xi_{N,x}) (d\eta_{N})_x =0.
\end{aligned}
\end{equation*}
 As each vector space
$T^*_x N_\Gamma$ is generated by $\psi_{\Gamma,x}(Z^1)$ and $\eta_{N,x}$, it is
enough to show that $\beta \circ [\varphi_N,\varphi_N]_{FN} + 2 \beta(\xi_N)
d\eta_N
=0$ just for $\beta \in \psi_\Gamma(Z^1)$ and for $\beta = \eta_N$. It is
easy to check that for all $\beta \in \Omega^1(N_\Gamma)$, one has
\begin{equation}
\label{betaphinphinfn}
\begin{aligned}
\beta \circ [\varphi_N,\varphi_N]_{FN} & =
2(-d ( (\varphi_N^t)^2\beta) -  (\wedge^2\!\varphi_N^t) (d\beta)
+ i_{\varphi_N}d (\varphi_N^t \beta)).
\end{aligned}
\end{equation}
Suppose $\beta \in \psi_\Gamma(Z^1)$. Then by the definition of $\varphi_N^t$ also
$\varphi_N^t\beta$ and $(\varphi_N^t)^2\beta$ are elements of
$\psi_\Gamma(Z^1)$.
As all the elements in $Z^1$ are closed and $\psi_\Gamma$ commutes with
the
differentials, we get that all three forms $\beta$, $\varphi_N^t\beta$, and
$(\varphi_N^t)^2 \beta$ are closed. Hence by~\eqref{betaphinphinfn}, $\beta
\circ [\varphi_N,\varphi_N]_{FN} =0$. From~\Cref{projectable}\eqref{psigammaz1}
and the definition of $\xi_N$, we get that
$\beta(\xi_N) =0$. Thence $\beta \circ ( [\varphi_N,\varphi_N]_{FN} + 2 d\eta_N \otimes
\xi_N ) =0$ for all $\beta \in \psi_\Gamma(Z^1)$.

Now take $\beta = \eta_N$. Using $\varphi_N^t \eta_N=0$ and \eqref{betaphinphinfn},
we get
\begin{equation*}
\begin{aligned}
\eta_N \circ [\varphi_N,\varphi_N]_{FN} = - 2(\wedge^2 \varphi_N^t) d\eta_N.
\end{aligned}
\end{equation*}
Hence, we have to check that
\begin{equation}
\label{wedgephidetaN}
\begin{aligned}
 (\wedge^2 \varphi_N^t)d\eta_N =   d\eta_N.
\end{aligned}
\end{equation}
Write $\psi$ for $\psi_\Gamma|_{{\liealgdual_\Gamma}}$. Then
$\psi_\Gamma|_{\extalg^2{\liealgdual_\Gamma} } = \wedge^2 \psi$. We have $d\eta_N = d\psi_\Gamma (\tileta) = \psi_\Gamma (d\tileta) =
\wedge^2\psi (d\tileta)$. As $\liealg_\Gamma$ is isomorphic to the Heisenberg
Lie algebra,  one gets that
$d\tileta \in \extalg^2 Z^1$.  Hence
\begin{equation*}
\begin{aligned}
(\wedge^2 (\varphi_N^t \circ \psi)) d\tileta = (\wedge^2 (\psi \circ \tau^{-1}
\circ \varphi^t \circ \tau) ) d\tileta.
\end{aligned}
\end{equation*}
We get that~\eqref{wedgephidetaN} is equivalent to
\begin{equation}
\label{psigammawedgetauphitaudetaheis}
\begin{aligned}
\psi_\Gamma \left( (\wedge^2 \tau^{-1} \circ \varphi^t \circ \tau) (d\tileta)
\right) = \psi_\Gamma(d\tileta).
\end{aligned}
\end{equation}
As $\psi_\Gamma$ is injective, as well as $\rho$ by \Cref{rhoinjective},
\eqref{psigammawedgetauphitaudetaheis} is equivalent to
\begin{equation*}
\begin{aligned}
\rho \left( (\wedge^2 \tau^{-1} \circ \varphi^t \circ \tau) (d\tileta)
\right) = \rho(d\tileta).
\end{aligned}
\end{equation*}
As
$\rho|_{Z^1 } = \tau$, we have $\rho|_{\extalg^2 Z^1} = \wedge^2 \tau$.
Thus the above equation becomes
\begin{equation}
\label{wedgephirhodetaheis}
\begin{aligned}
  (\wedge^2  \varphi^t) \rho (d\tileta) = \rho(d\tileta).
\end{aligned}
\end{equation}
From the characteristic property of $\tileta$, we get
\begin{equation*}
\begin{aligned}
\rho(d\tileta) & = d \rho(\tileta) =  d ( \eta + i_\varphi df ) = d\eta +
di_\varphi df.
\end{aligned}
\end{equation*}
Notice that $i_\varphi d\eta = 0$ by $\eqref{iphideta}$.
As $f$ is basic, we get
\begin{equation*}
\begin{aligned}
i_\varphi di_\varphi df & = \lie_\varphi^2 f + d i_{\varphi}^2 df = (d_B^c)^2
f + d ( df \circ \varphi^2) =  d (-df
+ df(\xi) \eta)   =  0.
\end{aligned}
\end{equation*}
Thus $i_\varphi \rho(d\tileta) =0$.
Since $d\tileta\in \extalg^2 Z^1$ and $\rho(Z^1)$ is a subset of basic
forms on $M$, we get that $\rho(d\tileta)$ is a basic form.

Now, if $\omega$ is a basic $2$-form on $M$ such that $i_\varphi \omega=0$, then
\mbox{$(\wedge^2 \varphi^t) \omega = \omega$}. Indeed, for any $X$, $Y\in TM$, we get
\begin{equation*}
\begin{aligned}
0 = (i_\varphi \omega) (\varphi X, Y) & = \omega (\varphi^2 X, Y) +
\omega( \varphi X, \varphi Y) \\ & = - \omega (X,Y) + \eta(X) i_\xi \omega(Y) + (\wedge^2
\varphi^t)(\omega) (X,Y)
\\& = -\omega(X,Y) + (\wedge^2 \varphi^t) (\omega) (X,Y).
\end{aligned}
\end{equation*}
This proves \eqref{wedgephirhodetaheis}.
\end{proof}
\begin{claim}
\label{hdftwisted}
The map $h$ is $df$-twisted.
\end{claim}
\begin{proof}
By the definition of $\xi_N$, we have $Th \circ \xi = \xi_N \circ h$. Next,
\begin{equation}
\label{etaNdftwisted}
\begin{aligned}
h^* \eta_N = h^* \psi_\Gamma (\tileta) =\rho(\tileta) = \eta + i_\varphi
(df).
\end{aligned}
\end{equation}

We have to check for every $x \in M$
the equality of maps from $T_x M$ to $T_{h(x)} N_\Gamma$
\begin{equation*}
\begin{aligned}
T_{x}h \circ ( \varphi_{x} + \xi_{M,x} \otimes
(df)_x) = \varphi_{N,x} \circ T_xh.
\end{aligned}
\end{equation*}
As $\psi_\Gamma(Z^1)_{h(x)} \oplus \left\langle \eta_{N,h(x)} \right\rangle =
T^*_{h(x)}(N_\Gamma)$ and $\psi_\Gamma(Z^1)_{h(x)} = \Ann (\xi_{N,h(x)})$, it is enough to show that for every $\alpha \in Z^1$
\begin{equation}
\label{hpsigammaalphaphi}
\begin{aligned}
(h^* \psi_\Gamma(\alpha)) \circ (\varphi + \xi\otimes df) = h^* ( \psi_\Gamma(\alpha) \circ \varphi_N)
\end{aligned}
\end{equation}
and that
\begin{equation}
\label{hetaNphi}
\begin{aligned}
(h^* \eta_N) \circ (\varphi + \xi \otimes df) = 0.
\end{aligned}
\end{equation}
For $\alpha \in Z^1$ and $x \in M$, we have $h^*\psi_\Gamma(\alpha) (\xi_x)=
\psi_\Gamma(\alpha) (\xi_{N,h(x)}) =0$.
Thus \eqref{hpsigammaalphaphi} is equivalent to
\begin{equation*}
\begin{aligned}
(h^* \psi_\Gamma(\alpha)) \circ \varphi = h^* ( \psi_\Gamma(\alpha) \circ
\varphi_N).
\end{aligned}
\end{equation*}
From the definition of $\varphi_N$, it follows that $\psi_\Gamma(\alpha) \circ
\varphi_N = \psi_\Gamma ( \tau^{-1} ( \tau(\alpha) \circ \varphi))$.
As $h^*\circ \psi_\Gamma = \rho$, we get
\begin{equation*}
\begin{aligned}
h^* (\psi_\Gamma (\alpha) \circ \varphi_N)  & = h^* (\psi_\Gamma ( \tau^{-1}
( \tau(\alpha) \circ \varphi))) = \rho \circ \tau^{-1} ( \rho(\alpha) \circ
\varphi) \\& = \rho(\alpha) \circ \varphi  = (h^* \psi_\Gamma(\alpha)) \circ
\varphi.
\end{aligned}
\end{equation*}
It is left to verify~\eqref{hetaNphi}. In the following computation we use that
$f$ is basic at the last step
\begin{equation*}
\begin{aligned}
(h^* \eta_N) \circ (\varphi &+ \xi \otimes df)  = (h^* \psi_\Gamma (\tileta))
\circ \varphi  + (h^*\eta_N)(\xi) df \\&  =
\rho(\tileta) \circ \varphi  + df = (\eta + i_\varphi df) \circ \varphi + df \\&  =
df \circ (-\id + \xi \otimes \eta) + df = 0.
\end{aligned}
\end{equation*}
\end{proof}
To show that $h$ is a diffeomorphism we
will use the map $h_f \colon M \times \R\to N_\Gamma\times \R$
defined by \eqref{hf}.  The map $h_f$ can be written
as the composition $\tilde{h}\circ \tilde{f}$, where the maps
$\tilde{f}$ and $\tilde{h}$ are given by
\begin{equation*}
\begin{aligned}
\tilde{f} \colon  M \times \R& \to M \times \R &\phantom{==}&& \tilde{h}\colon M\times \R &\to
N_\Gamma \times \R
\\  (x, t) & \mapsto (x , t+ f(x)) &\phantom{==}&& (x,t) & \mapsto ( h(x) , t).
\end{aligned}
\end{equation*}
The map $\tilde{f}$ is a diffeomorphism. The map $\tilde{h}$ is a
diffeomorphism if and only if $h$ is a diffeomorphism. Hence also $h_f$ is a
diffeomorphism if and only if $h$ is a diffeomorphism.

\begin{claim}
\label{hfsurjectiveproper}
The map $h_f$ is surjective, universally closed and proper.
\end{claim}
\begin{proof}
The map $h_f$ is surjective being a composition of two surjective maps.
Next, the map $h$ is proper since it is  a continuous map between compact topological
spaces. Given a Hausdorff topological space $X$ and a locally compact Hausdorff
space $Y$, a continuous map $X\to Y$ is proper if and only if it is universally closed.
Thus $h$ is universally closed, and, hence, also $\tilde{h}$ is universally
closed and, thus, proper. The same properties hold for $h_f$, as $\tilde{f}$ is a homeomorphism.
\end{proof}

By \Cref{hfholomorphicnew} and \Cref{hdftwisted}, the map $h_f$ is
holomorphic.

Recall that a continuous map $\psi\colon X \to Y$ between two topological spaces is called
\emph{finite} if it is closed and has finite fibers.
\begin{claim}
\label{hffinite}
The map $h_f\colon M \times \R\to N_\Gamma \times \R$ is finite.
\end{claim}
\begin{proof}
We already saw in \Cref{hfsurjectiveproper} that $h_f$ is  closed.  Thus it
is left to show that $h_f$ has finite fibers.
Fix $y\in N_\Gamma \times \R$.
 Since
$h_f$ is holomorphic, $h_f^{-1}(y)$ is a complex analytic subvariety
of $M\times \R$. By \Cref{hfsurjectiveproper}, $h_f$ is proper. Thus
$h_f^{-1}(y)$ is compact. Hence $h_f^{-1}(y)$ is a union of finitely many irreducible
complex subvarieties (cf. \cite[Sec. 9.2.2]{grauert2}).
By \Cref{exact}, the Kähler form on
$M \times \R$ is exact.
We will show in \Cref{exactkaehler}, that
every compact irreducible subvariety
  of a Kähler manifold with an
exact Kähler form is a point. Hence $h_f^{-1}(y)$ is a union of finitely many
points.
\end{proof}
\begin{lemma}
\label{exactkaehler}
Let $X$ be a Kähler manifold and $Z \subset X$ an irreducible compact complex analytic
subvariety in $X$. If the Kähler form of $X$ is exact, then $Z$ is a point.
\end{lemma}
\begin{proof}
By embedded Hironaka resolution of singularities for the pair $Z \subset X$,
there is a proper birational  holomorphic map $\pi \colon \widetilde{X}\to X$ of
complex manifolds with exceptional
locus $\Sigma$, such that the strict transform
\begin{equation*}
\begin{aligned}
\widetilde{ Z } = \overline{\pi^{-1}(Z \setminus \Sigma ) }
\end{aligned}
\end{equation*}
of $Z$ is a smooth complex subvariety of $\widetilde{X}$ and the restriction of $\pi$ to
$\widetilde{Z}$ is an immersion.
Denote this restriction by $\sigma$.
In our case, we have that $\widetilde{Z}$ is a compact submanifold of
$\widetilde{X}$. Indeed, $\widetilde{Z}$ is a closed subset of $\pi^{-1}(Z)$,
which  is compact since
$\pi$ is proper and $Z$ is compact.

Write $\omega $ for the Kähler form on
$X$. Since $\omega$ is exact, we have $\omega = d\alpha$ for some $\alpha \in \Omega^1(X)$.
The form $\sigma^*\omega$ is a Kähler form on $\widetilde{Z}$. Indeed, it is
obviously closed and it is positive, since for every nonzero $X \in T\widetilde{Z}$
\begin{equation*}
\begin{aligned}
\sigma^*\omega (JX, X) = \omega ( \sigma_*(JX), \sigma_* X ) = \omega (
J\sigma_* X , \sigma_* X)  = g( \sigma_* X , \sigma_* X)>0.
\end{aligned}
\end{equation*}
We used that $\sigma$ is an immersion in the last step.
As $\sigma^* \omega = d (\sigma^*\alpha)$, we get that $\widetilde{Z}$ is a
compact Kähler manifold with an exact Kähler form, which is possible only if
$\widetilde{Z}$ is a finite union of  points.
Hence $Z = \sigma(\widetilde{Z})$ is a finite union of points. Since $Z$ is irreducible, we
conclude that $Z$ is a point.
\end{proof}
Now we are ready to finish the proof of \Cref{main2}.
\begin{claim}
The map $h_f$ is a  biholomorphism. In particular, $h_f$ and $h$ are
diffeomorphisms.
\end{claim}
\begin{proof}
By~\cite[Sec. 9.3.3]{grauert2} a finite holomorphic surjection between irreducible
complex spaces is an analytic covering.
The map $h_f$ is surjective
by \Cref{hfsurjectiveproper}, finite by \Cref{hffinite}, and
holomorphic by~\Cref{hfholomorphicnew} and \Cref{hdftwisted}.
 Thus  $h_f$ is an analytic covering.
Hence there is a nowhere dense closed subset $T$ in $N_\Gamma \times \R$ such that the
induced map
\begin{equation*}
\begin{aligned}
h_T \colon  (M\times \R) \setminus h_f^{-1}(T) \to (N_\Gamma
\times \R) \setminus T
\end{aligned}
\end{equation*}
is
locally biholomorphic.
To complete the proof it is enough
to show
 that $h_T$ is a biholomorphism. Indeed, this will imply that $h_f$ is
a one-sheeted analytic covering and then by a result in~\cite[Sec.
8.1.2]{grauert2} the map $h_f$ is a bijection.
By~\cite[Corollary~8.6]{grauert}, every holomorphic bijection is a
biholomorphism. Hence, we will get that $h_f$ is a biholomorphism.

To show that $h_T$ is a biholomorphism it is enough to show that $h_T$ is
bijective.
Let $(y,b) \in R := (N_\Gamma \times \R) \setminus T$.
Then $(y,b)$ is a regular value of $h_T$ and thus a regular value of $h_f$.
Let $(x,a)$ be in the preimage of $(y,b)$.
As $(x,a)$ is a regular point of the holomorphic map $h_f$, we have
$\det T_{(x,a)} h_f >0$. But $\det T_{(x,a)}h_f = \det T_x h$, hence
$\det T_x h > 0$ for all $x \in h^{-1}(y)$.
In particular, $x$ is a regular point of $h$. Since $(x,a)$ was arbitrary, we
conclude that $y$ is a regular value of $h$.
By \Cref{hsurjective}, the map $h$ has degree $\pm 1$. Hence
\begin{equation*}
\begin{aligned}
\pm 1 = \deg(h) = \sum_{x \in h^{-1}(y)} \sign (\det T_x h) = |h^{-1}(y)|.
\end{aligned}
\end{equation*}
Therefore the number of points in $h^{-1}(y)$ is one. Let $x$ be the unique
point in $h^{-1}(y)$. Then $h_T^{-1}(y,b) = h_f^{-1}(y,b) = \{(x,b - f(x) \}$.
Hence $h_T$ is a bijection.
\end{proof}

\section*{Acknowledgment}
We are grateful to Hisashi Kasuya for suggesting that \Cref{solvable} should hold. Moreover, we would like to thank the anonymous referee for providing suggestions that
permitted us to clarify and simplify some arguments of our proof.

\bibliography{sas}
\bibliographystyle{plain}
\end{document}